\documentclass[a4paper,10pt]{amsart}

\usepackage[english]{babel}
\usepackage[utf8]{inputenc}

\usepackage{amssymb}
\usepackage{amsmath}
\usepackage{amsthm}
\usepackage{bbm}

\usepackage{enumerate}
\usepackage{tikz}
\usepackage{marginnote}

\newcommand{\bbC}{\mathbb{C}}
\newcommand{\bbN}{\mathbb{N}}

\newcommand{\bbR}{\mathbb{R}}

\newcommand{\bbZ}{\mathbb{Z}}
\newcommand{\calA}{\mathcal{A}}

\newcommand{\calL}{\mathcal{L}}
\newcommand{\calT}{\mathcal{T}}

\DeclareMathOperator{\id}{id}

\DeclareMathOperator{\re}{Re}

\newcommand{\sa}{{\operatorname{sa}}}

\DeclareMathOperator{\spb}{s}

\newcommand{\per}{{\operatorname{per}}}
\newcommand{\bnd}{{\operatorname{bnd}}}

\theoremstyle{definition}
\newtheorem{definition}{Definition}[section]
\newtheorem{remark}[definition]{Remark}

\theoremstyle{plain}
\newtheorem{proposition}[definition]{Proposition}

\newtheorem{theorem}[definition]{Theorem}

\numberwithin{equation}{section}

\begin{document}

\title[$C^*$-algebras and Perron--Frobenius theory]{A note on lattice ordered $C^*$-algebras and Perron--Frobenius theory}
\author{Jochen Gl\"uck}
\email{jochen.glueck@uni-ulm.de}
\address{Jochen Gl\"uck, Institute of Applied Analysis, Ulm University, 89069 Ulm, Germany}
\keywords{Lattice ordered $C^*$-algebra; Perron--Frobenius theory; completely positive semigroup; commutativity from order}
\subjclass[2010]{46L05; 46B40; 47B65}
\date{\today}
\begin{abstract}
	A classical result of Sherman says that if the space of self-adjoint elements in a $C^*$-algebra $\mathcal{A}$ is a lattice with respect to its canonical order, then $\mathcal{A}$ is commutative. We give a new proof of this theorem which shows that it is intrinsically connected with the spectral theory of positive operator semigroups. Our methods also show that some important Perron--Frobenius like spectral results fail to hold in any non-commutative $C^*$-algebra.
\end{abstract}
\maketitle

\section{Introduction}

Let us consider the space $\calA_\sa$ of self-adjoint elements of a $C^*$-algebra $\calA$. There is a canonical order on $\calA_\sa$ which is given by $a \le b$ if and only if $b-a$ is positive semi-definite. With respect to this order $\calA_\sa$ is an ordered real Banach space, meaning that the \emph{positive cone} $(\calA_\sa)_+ := \{a \in \calA_\sa| \, a \ge 0\}$ is closed, convex and invariant with respect to multiplication by scalars $\alpha \ge 0$ and that it fulfils $(\calA_\sa)_+ \cap -(\calA_\sa)_+ = \{0\}$. If $\calA$ is commutative, then it follows from Gelfand's representation theorem for commutative $C^*$-algebras that $\calA_\sa$ is actually lattice ordered, i.e.\ that all elements $a,b \in \calA_\sa$ have an \emph{infimum} (or \emph{greatest lower bound}) $a \land b$. In 1951 Sherman proved that the converse implication also holds, i.e.\ if $\calA_\sa$ is lattice ordered, then $\calA$ is commutative \cite[Theorems~1 and~2]{Sherman1951}.

Since then a great wealth of results has appeared which vary and generalise Sherman's theorem in several directions. We only quote a few of them: in \cite[Theorems~1 and~2]{Fukamiya1954} it is shown that if the ordered space $\calA_\sa$ is only assumed to have the Riesz decomposition property, then $\calA$ is commutative. In \cite[Theorem~2]{Archbold1974} it is proved that if for a fixed element $a \in \calA_\sa$ and all $b \in \calA_\sa$ the infimum $a \land b$ exists, then $\calA$ is commutative. In \cite[Corollary~11 and Theorem~12]{Kadison1951}, \cite[Theorem~3.7]{Archbold1972}, \cite[Theorem~3.5]{Cho-Ho1973} and \cite[Corollary~2.4]{Green1977} it is shown that if $\calA$ is in some sense too far from being commutative, then $\calA_\sa$ is even a so-called \emph{anti-lattice}, meaning that two elements of $a,b\in \calA_\sa$ have an infimum only if $a \le b$ or $b \le a$. In \cite{Curtis1958} Sherman's result is adapted to more general Banach algebras. The ``commutativity-from-order'' theme was also considered from a somewhat different viewpoint in \cite{Topping1965} and a more recent contribution to the theory can be found in \cite{Spain2001}.

It might not come as a surprise that the proofs for many of the above mentioned results use methods which are rather typical for the theory of operator algebras. The main goal of this note is to show that Sherman's theorem can also be proved by using only elementary properties of $C^*$-algebras if we instead employ the so-called \emph{Perron--Frobenius theory}; this notion is usually used to refer to the spectral theory of positive operators and semigroups, in particular on Banach lattices. We present our new proof of Sherman's result in Section~\ref{section:a-new-proof-of-shermans-theorem}.

Since our approach relies on non-trivial spectral theoretic results, we do not claim that our proof is simpler than the original one; our main point is that our proof employs completely different methods and thus sheds some new light on Sherman's theorem. In particular, the proof establishes a beautiful connection between commutativity of operator algebras and Perron--Frobenius theory. In Section~\ref{section:perron-frobenius-theory-on-c-star-algebras} we have a somewhat deeper look into this connection by proving that certain Perron--Frobenius type results can never hold on a non-commutative $C^*$-algebra.

\section{A new proof of Sherman's theorem} \label{section:a-new-proof-of-shermans-theorem}

In this section we give a new proof of Sherman's commutativity result by means of Perron--Frobenius spectral theory. For every Banach space $E$ we denote by $\calL(E)$ the space of bounded linear operators on $E$.

\begin{theorem}[Sherman] \label{thm:sherman}
	Let $\calA$ be a $C^*$-algebra and let the space $\calA_\sa$ of self-adjoint elements in $\calA$ be endowed with the canonical order. Then $\calA_\sa$ is lattice ordered if and only if $\calA$ is commutative.
\end{theorem}
\begin{proof}
	The implication ``$\Leftarrow$'' follows readily from Gelfand's representation theorem for commutative $C^*$-algebras. To prove the implication ``$\Rightarrow$'', assume that $\calA_\sa$ is lattice ordered. Fix $a \in \calA_\sa$. It suffices to show that $a$ commutes with every element $c \in \calA_\sa$. 
	
	For every $t \in \bbR$ we define an operator $T_t \in \calL(\calA_\sa)$ by $T_tc = e^{-ita}\,c\,e^{ita}$ (where the exponential function is computed in the unitization of $\calA$ in case that $\calA$ does not contain a unit). Obviously, the operator family $\calT := (T_t)_{t \in \bbR} \subseteq \calL(\calA_\sa)$ is a contractive positive $C_0$-group on $\calA_\sa$. The orbit $t \mapsto T_tc$ is differentiable for every $c \in \calA_\sa$ and its derivative at $t = 0$ is given by $-iac + ica$; hence, the $C_0$-group $\calT$ has a bounded generator $L \in \calL(\calA_\sa)$ which is given by $Lc  = -iac + ica$ for all $c \in \calA_\sa$. 
	
	Let $\sigma(L)$ denote the spectrum of (the complex extension of) $L$; then $\sigma(L)$ is contained in $i\bbR$ since the $C_0$-group $\calT$ is bounded and $\sigma(L)$ is non-empty and bounded since $L \in \calL(\calA_\sa)$. Now we use that $\calA_\sa$ is a lattice to show that $\sigma(L) = \{0\}$: since the norm is monotone on $\calA_\sa$ (meaning that we have $\|a\| \le \|b\|$ whenever $0 \le a \le b$, see \cite[Theorem~2.2.5(3)]{Murphy1990}), we can find an equivalent norm on $\calA_\sa$ which renders it a Banach lattice (see Section~\ref{section:perron-frobenius-theory-on-c-star-algebras} for a description of this norm). However, if an operator $L$ with spectral bound $0$ generates a bounded, eventually norm continuous and positive $C_0$-semigroup on a Banach lattice, then it follows from Perron--Frobenius theory that $0$ is a dominant spectral value of $L$, i.e.\ that we have $\sigma(L) \cap i\bbR = \{0\}$; see \cite[Corollary~C-III-2.13]{Arendt1986}. Hence, $\sigma(L) = \{0\}$. 
	
	Since $L$ generates a bounded $C_0$-group it now follows from the semigroup version of Gelfand's $T = \id$ theorem \cite[Corollary~4.4.12]{Arendt2011} that $\calT$ is trivial, i.e.\ that $T_t$ is the identity operator on $\calA_\sa$ for all $t \in \bbR$. Hence $L = 0$, so $0 = Lc = -iac + ica$ for all $c \in \calA_\sa$. This shows that $a$ commutes with all elements of $\calA_\sa$.
\end{proof}

In the above proof we quoted a result of Perron--Frobenius theory for positive operator semigroups from \cite[Corollary~C-III-2.13]{Arendt1986}, and this result is in turn based on a rather deep theorem in \cite[Theorem~C-III-2.10]{Arendt1986}. However, we only needed a simple special case since the generator of our semigroup is bounded. Let us demonstrate how this special case can be treated without employing the entire machinery of Perron--Frobenius theory:

\begin{proposition} \label{prop:perron-frobenius-for-bounded-generator}
	Let $E$ be a Banach lattice and let $L \in \calL(E)$ be an operator with spectral bound $s(L) = 0$. If $e^{tL}$ is positive for every $t \ge 0$, then $\sigma(L) \cap i \bbR = \{0\}$.
\end{proposition}
\begin{proof}
	Assume that $e^{tL} \ge 0$ for all $t \ge 0$. Since $E$ is a Banach lattice it follows that $L + \|L\| \ge 0$, see \cite[Theorem~C-II-1.11]{Arendt1986}. One can prove by a simple resolvent estimate that the spectral radius of a positive operator on a Banach lattice is contained in the spectrum, see e.g.~\cite[Proposition~V.4.1]{Schaefer1974}. Hence $r(L + \|L\|) \in \sigma(L + \|L\|)$, so we conclude that the spectral bound $\spb(L+\|L\|) = \|L\|$ coincides with the spectral radius $r(L + \|L\|)$. Thus, $\sigma(L + \|L\|) \cap (\|L\| + i\bbR) = \{\|L\|\}$, which proves the assertion.
\end{proof}

\begin{remark} \label{rem:gelfands_t-id-theorem}
	In our proof of Theorem~\ref{thm:sherman} we used another non-trivial result, namely the semigroup version of Gelfand's $T = \id$ theorem. Let us note that we can instead use Gelfand's $T = \id$ theorem for single operators: if we know that $\sigma(L) = \{0\}$ and that the group $(e^{tL})_{t \in \bbR}$ is bounded, then it follows from the spectral mapping theorem for the holomorphic functional calculus (or for eventually norm continuous semigroups) that $\sigma(e^{tL}) = \{1\}$ for every $t \in \bbR$; since every $e^{tL}$ is doubly power bounded, we conclude from Gelfand's $T = \id$ theorem for single operators (see e.g.~\cite[Theorem~B.17]{Engel2000}) that $e^{tL} = \id_{\calA_\sa}$ for all $t \in \bbR$.
\end{remark}

Despite what was said in Proposition~\ref{prop:perron-frobenius-for-bounded-generator} and Remark~\ref{rem:gelfands_t-id-theorem}, our proof of Sherman's theorem still relies on Gelfand's $T = \id$ theorem for single operators, which is a non-trivial result. Thus, our proof is not elementary; yet, all its non-elementary ingredients are essentially independent of $C^*$-algebra theory.

\section{Perron--Frobenius theory on $C^*$-algebras} \label{section:perron-frobenius-theory-on-c-star-algebras}

Our proof of Theorem~\ref{thm:sherman} suggests that certain Perron--Frobenius type spectral results can only be true on commutative $C^*$-algebras. Let us discuss this in a bit more detail: indeed, it was demonstrated in \cite[pp.\,387--388]{Arendt1986} and \cite[Section~4]{Luczak2010} by means of concrete examples that typical Perron--Frobenius results which are true on Banach lattices fail in general on non-commutative $C^*$-algebras. On the other hand, some results can even by shown in the non-commutative setting if one imposes additional assumptions (such as irreducibility and complete positivity) on the semigroup or the operator under consideration. 

For single matrices and operators, such results can for instance be found in \cite{Evans1978, Groh1981, Groh1982, Groh1983}. For time-continuous operator semigroups we refer for example to \cite[Section~D-III]{Arendt1986}; the papers \cite{Albeverio1978, Luczak2010} contain results for both the single operator and the time-continuous case, and in \cite[Section~6.1]{Batkai2012} Perron--Frobenius type results on $W^*$-algebras are proved by a different approach, using a version of the so-called Jacobs--DeLeeuw--Glicksberg decomposition. An example for an application of Perron--Frobenius theory on $C^*$-algebras to the analysis of quantum systems can be found in \cite[Theorem~2.2]{Jaksic2014}.

In this section we are headed in a different direction: we use the idea of our proof of Theorem~\ref{thm:sherman} to show that certain Perron--Frobenius type results are \emph{never} true on non-commutative $C^*$-algebras. 

To state our result we need a bit of notation. Some of the following notions have already been used tacitly above, but to avoid any ambiguity in the formulation of the next theorem, it is important to recall them explicitly here. If $E$ is a Banach space, then we denote by $\calL(E)$ the space of all bounded linear operators on $E$; the dual space of $E$ is denoted by $E'$ and the adjoint of an operator $L \in \calL(E)$ is denoted by $L' \in \calL(E')$. 

By an \emph{ordered Banach space} we mean a tuple $(E,E_+)$, often only denoted by $E$, where $E$ is a real Banach space, and $E_+ \subset E$ is a closed and pointed cone in $E$, meaning that $E_+$ is closed, that $\alpha E_+ + \beta E_+ \subseteq E_+$ for all $\alpha,\beta \in [0,\infty)$ and $E_+ \cap -E_+ = \{0\}$. On an ordered Banach space there is a canonical order relation $\le$ which is given by ``$x \le y$ iff $y - x \in E_+$''. The cone $E_+$ is called \emph{generating} if $E = E_+ - E_+$ and it is called \emph{normal} if there exists a constant $c > 0$ such that $\|x\| \le c\|y\|$ whenever $0 \le x \le y$. 

If $\calA$ is a $C^*$-algebra, then the space $\calA_\sa$ of self-adjoint elements in $\calA$ is usually endowed with the cone $(\calA_\sa)_+ := \{a \in \calA_\sa: \; \sigma(a) \subseteq [0,\infty)\}$; thus one obtains the canonical order on $\calA_\sa$ that we already considered in Theorem~\ref{thm:sherman}. Note that the cone $(\calA_\sa)_+$ is normal (since $\|a\| \le \|b\|$ whenever $0 \le a \le b$, see \cite[Theorem~2.2.5(3)]{Murphy1990}) and generating \cite[p.\,45]{Murphy1990}.

Let $E$ be an ordered Banach space. For all $x,y \in E$ the \emph{order interval} $[x,y]$ is given by $[x,y] := \{f \in E: \, x \le f \le y\}$. The space $E$ is said to have the \emph{Riesz decomposition property} if $[0,x] + [0,y] = [0,x+y]$ for every $x,y \in E_+$. If $E$ is lattice ordered, then the cone $E_+$ is generating and $E$ has the Riesz decomposition property \cite[Corollary~1.55]{Aliprantis2007}, but the converse is not in general true. 

If $E$ is an ordered Banach space with generating and normal cone and if the induced order makes $E$ a vector lattice, then $\|x\|_{\operatorname{BL}} := \sup_{0 \le z \le |x|} \|z\|$ defines an equivalent norm on $E$ which renders it a Banach lattice (use the  Riesz decomposition property of $E$ to see that $\|\cdot\|_{\operatorname{BL}}$ satisfies the triangle inequality and use e.g.\ \cite[Theorem~2.37(3)]{Aliprantis2007} to see that the norm $\|\cdot\|_{\operatorname{BL}}$ is indeed equivalent to the original norm).

Let $E$ be an ordered Banach space. An operator $L \in \calL(E)$ is called \emph{positive} if $LE_+ \subseteq E_+$; this is denoted by $L \ge 0$. By $E'_+$ we denote the set of all positive functionals on $E$; here, a functional $x' \in E'$ is called \emph{positive} if $\langle x', x\rangle \ge 0$ for all $x \in E_+$. It follows from the Hahn--Banach separation theorem that a vector $x \in E$ is contained in $E_+$ if and only if $\langle x', x\rangle \ge 0$ for all $x' \in E'_+$; hence, an operator $L \in \calL(E)$ is positive if and only if $L'E'_+ \subseteq E'_+$. If the positive cone in $E$ is generating, then $E'_+ \cap - E'_+ = \{0\}$ and thus, the dual space $E'$ is also an ordered Banach space.

Let $E$ be an ordered Banach space. We call an operator $L \in \calL(E)$ \emph{quasi-positive} if $L + \alpha \id_E \ge 0$ for some $\alpha \ge 0$, we call it \emph{exponentially positive} if $e^{tL} \ge 0$ for all $t \in [0,\infty)$ and and we call it \emph{cross positive} if $\langle x', Lx \rangle \ge 0$ for all $x \in E_+$ and all $x' \in E'_+$ which fulfil $\langle x', x \rangle = 0$. We point out, however, that those properties are sometimes named differently in the literature.

If $A: E \supseteq D(A) \to E$ is a closed linear operator on a real Banach space $E$, then $\sigma(A)$ denotes the spectrum of the complex extension of $A$ to some complexification of $E$; similarly, all other notions from spectral theory are understood to be properties of the complex extension of $A$. The \emph{spectral bound} of $A$ is the number $s(A) := \sup\{\re \lambda: \, \lambda \in \sigma(A)\} \in [-\infty,\infty]$ and if $s(A) \in \bbR$, then the \emph{boundary spectrum} if $A$ is defined to be the set $\sigma_{\bnd}(A) := \{\lambda \in \sigma(A): \; \re \lambda = s(A)\}$. If $L \in \calL(E)$, then the \emph{spectral radius} if $L$ is denoted by $r(L)$ and the \emph{peripheral spectrum} of $L$ is defined to be the set $\sigma_\per(L) := \{\lambda \in \sigma(L): \; |\lambda| = r(L)\}$.

Many results in Perron--Frobenius theory ensure the \emph{cyclicity} of parts of the spectrum; here, a set $S \subseteq \bbC$ is called \emph{cyclic} if $re^{i\theta} \in S$ ($r \ge 0$, $\theta \in \bbR$) implies that $re^{in\theta} \in S$ for all integers $n \in \bbZ$. The set $S$ is called \emph{additively cyclic} if $\alpha + i\beta \in S$ ($\alpha, \beta \in \bbR$) implies that $\alpha + in\beta \in S$ for all integers $n \in \bbZ$. 

An important result in Perron--Frobenius theory states that on any given Banach lattice $E$, every positive power bounded operator with spectral radius $1$ has cyclic peripheral spectrum; in fact, an even somewhat stronger result is true, see \cite[Theorem~V.4.9]{Schaefer1974}. However, it is still an open problem whether \emph{every} positive operator on a Banach lattice has cyclic peripheral spectrum; see \cite{Glueck2015, Glueck2016} for a detailed discussion of this and for some recent progress in this question. An analogues result for $C_0$-semigroups says that if the generator of a positive and bounded $C_0$-semigroup has spectral bound $0$, then its boundary spectrum is additively cyclic; see \cite[Theorem~C-III-2.10]{Arendt1986} for a slightly stronger result. 

Theorem~\ref{thm:perron-frobenius-on-c-star-algebras} below shows that the above mentioned results are \emph{never} true an a non-commutative $C^*$-algebra. In fact, they are not even true for \emph{completely positive} operators and semigroups on such spaces; recall that a linear operator $L \in \calL(\calA_\sa)$, whose complex extension to $\calA$ is denoted by $L_\bbC$, is called \emph{completely positive} if the operator $L_\bbC \otimes \id_{\bbC^{k \times k}} \in \calL(\calA \otimes \bbC^{k \times k})$ (more precisely: its restriction to the self-adjoint part of the $C^*$-algebra $\calA \otimes \bbC^{k \times k}$) is positive for every $k \in \bbN_0$; see \cite[Section~II.6.9]{Blackadar2006} for some details. We call a $C_0$-semigroup $(e^{tA})_{t \ge 0}$ on $\calA_\sa$ \emph{completely positive} if the operator $e^{tA}$ is completely positive for every $t \ge 0$. Completely positive operators play an important role in quantum physics. Here we are going to show that the above mentioned Perron--Frobenius type results are not even true for completely positive operators and semigroups if the $C^*$-algebra $\calA$ is non-commutative. To show this we only need the simple fact that for every $a \in \calA$ the operator $L \in \calL(\calA_\sa)$ given by $Lc = a^*ca$ is completely positive. 

\begin{theorem} \label{thm:perron-frobenius-on-c-star-algebras}
	Let $\calA$ be a $C^*$-algebra and let the space $\calA_\sa$ of self-adjoint elements in $\calA$ be endowed with the canonical order. The following assertions are equivalent:
	\begin{enumerate}[\upshape (i)]
		\item $\calA$ is commutative.
		\item $\calA_\sa$ is lattice ordered.
		\item $\calA_\sa$ has the Riesz decomposition property. 
		\item Every cross positive operator in $\calL(\calA_\sa)$ is quasi-positive. 
		\item Every exponentially positive operator in $\calL(\calA_\sa)$ is quasi-positive.
		\item Every power bounded positive operator in $\calL(\calA_\sa)$ with spectral radius $1$ has cyclic peripheral spectrum.
		\item Every power bounded completely positive operator in $\calL(\calA_\sa)$ with spectral radius $1$ has cyclic peripheral spectrum.
		\item If the generator $A$ of a bounded positive $C_0$-semigroup on $\calA_\sa$ has spectral bound $0$, then its boundary spectrum is additively cyclic. 
		\item If the generator $A$ of a bounded completely positive $C_0$-semigroup on $\calA_\sa$ has spectral bound $0$, then its boundary spectrum is additively cyclic. 
		\item Every norm continuous, bounded and completely positive $C_0$-group $(e^{tL})_{t \in \bbR}$ on $\calA_\sa$ is trivial, i.e.\ it fulfils $e^{tL} = \id_{\calA_\sa}$ for all $t \in \bbR$.
	\end{enumerate}
\end{theorem}
\begin{proof}
	We show ``(i) $\Rightarrow$ (ii) $\Rightarrow$ (iv) $\Rightarrow$ (v) $\Rightarrow$ (x) $\Rightarrow$ (i)'' as well as ``(ii) $\Rightarrow$ (iii) $\Rightarrow$ (v)'', ``(ii) $\Rightarrow$ (vi) $\Rightarrow$ (vii) $\Rightarrow$ (x)'' and ``(ii) $\Rightarrow$ (viii) $\Rightarrow$ (ix) $\Rightarrow$ (x)''.
	
	``(i) $\Rightarrow$ (ii)'' This follows from Gelfand's representation theorem for commutative $C^*$-algebras.
	
	``(ii) $\Rightarrow$ (iv)''. Since $\calA_\sa$ is lattice ordered and since the positive cone in $\calA_\sa$ is normal, there is an equivalent norm on $\calA_\sa$ which renders it a Banach lattice. Hence, it follows from \cite[Theorem~C-II-1.11]{Arendt1986} that every cross positive operator on $\calA_\sa$ is quasi-positive.
	
	``(iv) $\Rightarrow$ (v)'' Assume that (iv) is true and let $L \in \calL(\calA_\sa)$ be exponentially positive. Due to (iv) it suffices to show that $L$ is cross positive, so let $0 \le a \in \calA_\sa$ and $0 \le \varphi \in \calA_\sa'$ such that $\langle \varphi, a \rangle = 0$. Then we obtain
	\begin{align*}
		\langle \varphi, La \rangle = \lim_{t \downarrow 0} \frac{\langle \varphi, (e^{tL} - \id_{\calA_\sa})a \rangle}{t} = \lim_{t \downarrow 0} \frac{\langle \varphi, e^{tL}a \rangle}{t} \ge 0.
	\end{align*}
	Hence, $L$ is cross positive and it now follows from (iv) that $L$ is quasi-positive. 
	
	``(v) $\Rightarrow$ (x)'' Suppose that (v) is true, let $L \in \calL(\calA_\sa)$ and assume that the group $(e^{tL})_{t \in \bbR}$ is bounded and completely positive. We argue as in Remark~\ref{prop:perron-frobenius-for-bounded-generator}: we have $\emptyset \not= \sigma(L) \subseteq i\bbR$ and we know from (v) that $L + \alpha \id_{\calA_\sa}$ is positive for some $\alpha \ge 0$. Since the positive cone in $\calA_\sa$ is generating and normal, it follows that the spectral radius of the operator $L + \alpha \id_{\calA_\sa}$ is contained in its spectrum \cite[Section~2.2 in the Appendix]{Schaefer1999}; hence, the spectral radius equals the spectral bound $s(L + \alpha \id_{\calA_\sa})$, which is in turn equal to $\alpha$. Thus, $\sigma(L + \alpha \id_{\calA_\sa}) = \sigma(L + \alpha \id_{\calA_\sa}) \cap (i\bbR + \alpha) = \{\alpha\}$, so we conclude that $\sigma(L) = \{0\}$. It now follows from the semigroup analogue of Gelfand's $T = \id$ theorem \cite[Corollary~4.4.12]{Arendt2011} that $e^{tL} = \id_{\calA_\sa}$ for all $t \in \bbR$.
	
	``(x) $\Rightarrow$ (i)'' Let us argue as in the proof of Theorem~\ref{thm:sherman}: fix $a \in \calA_\sa$ and define $T_t \in \calL(\calA_\sa)$ by $T_tc = e^{-ita}ce^{ita}$ for every $t \in \bbR$ (where the exponential function is computed in the unitization of $\calA$ if $\calA$ does not contain a unit itself). Then $(T_t)_{t \in \bbR}$ is a bounded, completely positive and operator norm continuous group on $\calA_\sa$; its generator $L \in \calL(\calA_\sa)$ is given by $Lc = -iac + ica$ for all $c \in \calA_\sa$. It follows from (x) that $L = 0$ and hence, $a$ commutes with every $c \in \calA_\sa$, which in turn proves that $\calA$ is commutative.
	
	``(ii) $\Rightarrow$ (iii)'' Every vector lattice has the Riesz decomposition property, see e.g.\ \cite[Corollary~1.55]{Aliprantis2007}.
	
	``(iii) $\Rightarrow$ (v)'' Suppose that $\calA_\sa$ has the Riesz decomposition property. Since the cone in $\calA_\sa$ is normal, it follows that the dual space $\calA_\sa'$ is a vector lattice \cite[Theorem~2.47]{Aliprantis2007}. Moreover, the cone in $\calA_\sa'$ is normal, too \cite[Theorem~2.42]{Aliprantis2007}, so $\calA_\sa'$ is a Banach lattice with respect to an equivalent norm. If $L \in \calL(\calA_\sa)$ is exponentially positive, then so is its adjoint $L' \in \calL(\calA_\sa')$ and hence, $L'$ is quasi-positive according to \cite[Theorem~C-II-1.11]{Arendt1986}. This in turn implies that $L$ is quasi-positive itself.
	
	``(ii) $\Rightarrow$ (vi)'' If $\calA_\sa$ is lattice ordered, then it is a Banach lattice with respect to some equivalent norm (since the cone in $\calA_\sa$ is normal). Hence, assertion (vi) follows from Perron--Frobenius theory of single operators on Banach lattices, see \cite[Theorem~V.4.9]{Schaefer1974}.
	
	``(vi) $\Rightarrow$ (vii)'' This is obvious.
	
	``(vii) $\Rightarrow$ (x)'' Suppose that (vii) is true and that $L \in \calL(\calA_\sa)$ generates a completely positive and bounded group $(e^{tL})_{t \in \bbR}$. Then the spectrum of $L$ is bounded and contained in the imaginary axis. Applying the spectral mapping theorem for small $t>0$ we conclude that the peripheral spectrum of $e^{tL}$ can only be cyclic if it is contained in $\{1\}$. Hence, $\sigma(L) = \{0\}$, so assertion (x) follows from the semigroup analogue of Gelfand's $T = \id$ theorem \cite[Corollary~4.4.12]{Arendt2011}.
	
	``(ii) $\Rightarrow$ (viii)'' Assume that $\calA_\sa$ is lattice ordered. We use again that $\calA_\sa$ is then a Banach lattice for some equivalent norm. Hence, assertion (viii) follows from Perron--Frobenius theory for positive $C_0$-semigroups on Banach lattices, see \cite[Theorem~C-III-2.10]{Arendt1986}. 
	
	``(viii) $\Rightarrow$ (ix)'' This is obvious.
	
	``(ix) $\Rightarrow$ (x)'' If (ix) holds and if $L \in \calL(\calA_\sa)$ generates a completely positive and bounded group $(e^{tL})_{t \in \bbR}$, then it follows from (ix) that $\sigma(L) = \{0\}$. Hence we can again employ the semigroup analogue of Gelfand's $T = \id$ theorem \cite[Corollary~4.4.12]{Arendt2011} to conclude that (x) holds.
\end{proof}

The above proof shows that many (not all) of the implications in Theorem~\ref{thm:perron-frobenius-on-c-star-algebras} are also true on general ordered Banach spaces with generating and normal cone. However, since we are mainly interested in $C^*$-algebras in this note, we omit a detailed discussion of this.

\subsection*{Acknowledgement}

During his work on this article, the author was supported by a scholarship of the ``Lan\-des\-gra\-duier\-ten\-f\"or\-der\-ung Baden-W\"urttemberg'' (grant number 1301 LGFG-E).

%
\bibliographystyle{plain}
\bibliography{literature}

\end{document}